\documentclass[envcountsame]{llncs}
\usepackage[hidelinks]{hyperref}
\usepackage{amsmath,amsfonts,proof}
\spnewtheorem{defn}[theorem]{Definition}{\bfseries}{\rmfamily}
\spnewtheorem{prob}[theorem]{Open Problem}{\bfseries}{\rmfamily}

\def\dem#1{\fbox{$#1$}}
\def\val#1#2#3#4{\mathrm{val}_{#1,#2,#3}(#4)}
\def\sat#1#2#3#4{#1,#2,#3 \models #4}
\def\M{\mathfrak{M}}
\let\lif\rightarrow
\let\liff\leftrightarrow
\let\seq\equiv
\def\eps{\ensuremath{\varepsilon}}
\def\mrm#1{\ensuremath{\mathrm{#1}}}
\def\card#1{\left|#1\right|}
\def\deg#1{\mathrm{deg}(#1)}
\def\rk#1{\mathrm{rk}(#1)}

\def\Lea{L_{\eps\forall}}

\def\Trm{\mrm{Trm}}
\def\Frm{\mrm{Frm}}
\def\Var{\mrm{Var}}

\def\Typ{\mrm{Typ}}
\def\FV{\mrm{FV}}
\def\meps#1#2{\ensuremath{\varepsilon_#1\,#2}}
\def\st#1#2#3{\ensuremath{#1[#2/#3]}}
\def\ST#1#2#3{\ensuremath{#1\{#2/#3\}}}
\def\ext{\mathrm{ext}}
\def\inst#1{#1^\textrm{inst}}

\def\EC{\mrm{EC}}
\def\ECq{\ensuremath{\mathrm{EC}^=}}
\def\PC{\ensuremath{\mathrm{EC}_\forall}}
\def\ECe{\ensuremath{\mathrm{EC}_\eps}}
\def\ECeq{\ensuremath{\mathrm{EC}_\eps^=}}
\def\ECex{\ensuremath{\mathrm{EC}_\eps^\ext}}
\def\PCe{\ensuremath{\mathrm{EC}_{\eps\forall}}}

\newcommand*{\proves}[2][{}]{\mathrel{\vdash^{#1}_{#2}}}
\newcommand*{\nproves}[1]{\mathrel{\not\vdash_{#1}}}

\begin{document}
\title{Semantics and Proof Theory of the\\ Epsilon Calculus}
\titlerunning{Semantics and Proof Theory of the Epsilon Calculus}

\author{\href{http://ucalgary.ca/rzach}{Richard Zach}\thanks{Research supported by the Natural Sciences and Engineering Research Council.}}
\institute{Department of Philosophy, University of Calgary, Canada\\
  \email{rzach@ucalgary.ca}}

\maketitle

\begin{abstract}
The epsilon operator is a term-forming operator which replaces
quantifiers in ordinary predicate logic. The application of this
undervalued formalism has been hampered by the absence of well-behaved
proof systems on the one hand, and accessible presentations of its
theory on the other. One significant early result for the original
axiomatic proof system for the \eps-calculus is the first epsilon
theorem, for which a proof is sketched.  The system itself is
discussed, also relative to possible semantic interpretations. The
problems facing the development of proof-theoretically well-behaved
systems are outlined.
\end{abstract}

\section{Introduction}

A formalism for logical choice operators has long been available in
the form of Hilbert's epsilon calculus.  The epsilon calculus is one
of the first formal systems of first-order predicate logic. It was
introduced in 1921 by David Hilbert \cite{Hilbert:22a}, who proposed
to use it for the formalization and proof theoretical investigation of
mathematical systems. In the epsilon calculus, a term-forming operator
\eps{} is used, the intuitive meaning of which is an indefinite
choice function: $\meps x {A(x)}$ is some $x$ which satisfies $A(x)$
if $A(x)$ is satisfied at all, and arbitrary otherwise.  Quantifiers
can then be defined, e.g., $(\exists x)A(x)$ as $A(\epsilon_x A(x))$.

The epsilon calculus and proof theoretic methods developed for it,
such as the so-called epsilon substitution method, have mainly been
applied to the proof theoretic analysis of mathematical systems of
arithmetic and analysis (especially in work by Ackermann, Mints,
Arai). (See \cite{AvigadZach:2002} for a survey of the epsilon
calculus and its history.) Despite its long history and manifold uses,
the epsilon calculus as a logical formalism in general is not
thoroughly understood, yet its potential for applications in logic and
other areas, especially linguistics and computer science, has by far
not been fully explored.

There are various options for definitions of semantics of the epsilon
operator. The choice of $\meps x {A(x)}$ may be extensional (i.e.,
depend only on the set of $x$ which satisfy $A(x)$; this definition
validates the so-called axiom of \eps-extensionality), it may be
intensional (i.e., depend also on $A(x)$ itself;
\eps-extensionality fails), and it may be completely
indeterministic (i.e., different occurrences of the same \eps-term
$\meps x {A(x)}$ may select different witnesses for $A(x)$).  The
first and third versions have been investigated by Blass and Gurevich
\cite{BlassGurevich:2000:JSL}. These different semantics result in
different expressive power (in particular, over finite models), and
are characterized by different formalizations. Below we present the
first two versions of the semantics of the \eps-calculus and sketch
completeness results.

The very beginnings of proof theory in the work of Hilbert and his
students consisted in the proof theoretic study of axiom systems for
the \eps-calculus. One of the most significant results in this
connection are the epsilon theorems.  It plays a role similar to
Gentzen's midsequent theorem in the proof theory of the sequent
calculus: it yields a version of Herbrand's Theorem. In fact, it was
used to give the first correct proof of Herbrand's theorem (Hilbert
and Bernays \cite{HB70}).  In a simple formulation, the theorem states
that if an existential formula $(\exists x)A(x)$ (not containing \eps)
is derivable in the epsilon calculus, then there are terms $t_1$,
\dots, $t_n$ so that a (Herbrand-) disjunction $A(t_1) \lor \ldots
\lor A(t_n)$ is derivable in propositional logic.  The proof gives a
constructive procedure that, given a derivation of $(\exists x)A(x)$,
produces the corresponding Herbrand disjunction. An analysis of this
proof (see \cite{MoserZach:06}) gives a hyper-exponential bound on the
length of the Herbrand disjunction in the number of critical formulas
occurring in the proof.  The bound is essentially optimal, since it is
known from work by Orevkov and Statman that the length of Herbrand
disjunctions is hyper-exponential in the length of proofs of the
original existential formula (this is the basis for familiar speed-up
theorems of systems with cut over cut-free systems). In
section~\ref{epsthm} we prove the first epsilon theorem with identity,
along the lines of Bernays's proof.

A general proof theory of the epsilon calculus requires formal systems
that are more amenable to proof-theoretic investigations than the
Hilbert-type axiomatic systems studied in the Hilbert school.
Although some sequent systems for the epsilon calculus exist, it is
not clear that they are the best possible formulations, nor have their
proof-theoretic properties been investigated in depth.
Maehara's~\cite{Maehara:55} and Leisenring's \cite{Leisenring:1969}
systems were not cut-free complete. Yasuhara \cite{Yasuhara:82}
studied a cut-free complete system, but only gave a semantic
cut-elimination proof.  Section~\ref{proofth} surveys these and other
systems, and highlights some of the difficulties in developing a
systematic proof theory on the basis of them. Proof-theoretically
suitable formalisms for the \eps-calculus are still a desideratum for
applications of the epsilon calculus.

The classical \eps-calculus is usually investigated as a
proof-theoretic formalism, and no systematic study of the model theory
of epsilon calculi other than Asser's classic \cite{Asser:1957}
exists.  However, Abiteboul and Vianu \cite{AbiteboulVianu:91}, Blass
and Gurevich \cite{BlassGurevich:2000:JSL}, and Otto \cite{Otto:00}
have studied the model theory of choice operators in the context of
finite model theory and database query languages. And applications of
choice operators to model definite and indefinite noun phrases in
computational linguistics Meyer Viol \cite{MeyerViol:95} and von
Heusinger \cite{Heusinger:00,Heusinger:04} have led to the definition
of indexed epsilon calculus by Mints and Sarenac
\cite{MintsSarenac:2003}.

With a view to applications, it is especially important to develop the
semantics and proof theory of epsilon operators in non-classical
logics.  Of particular importance in this context is the development
of epsilon calculi for intuitionistic logic, not least because this is
the context in which the epsilon calculus can and has been applied in
programming language semantics.  Some work has been done on
intuitionistic $\eps$-calculi (e.g., Bell \cite{Bell:93a}, DeVidi
\cite{DeVidi:95}, Meyer Viol \cite{MeyerViol:95}, Mints
\cite{Mints:77}), but there are still many important open questions.
The straightforward extensions of intuitionistic logic by epsilon
operators are not conservative and result in intermediate logics
related to G\"odel logic.  Meyer Viol \cite{MeyerViol:95} has proposed
a conservative extensions of intuitionistic logic by epsilon operators
which warrants further study.

\section{Syntax and Axiomatic Proof Systems}

\begin{defn}
  The language of of the elementary calculus~$L_\EC^=$ contains the
  usual logical symbols (variables, function and predicate
  symbols, $=$).  A subscript $\eps$ will indicate the presence of the
  symbol~\eps, and $\forall$ the presence of the quantifiers $\forall$
  and $\exists$.  The \emph{terms}~$\Trm$ and \emph{formulas}~$\Frm$
  of~$\Lea$ are defined as usual, but simultaneously, to include:
\begin{quote}
 If $A$ is a formula in which $x$ has a free occurrence but no bound
 occurrence, then $\meps x A$ is a term, and all occurrences of $x$ in
 it are bound.
\end{quote}
If $E$ is an expression (term or formula), then $\FV(E)$ is the set of
variables which have free occurrences in~$E$.
\end{defn}

When $E$, $E'$ are expressions (terms or formulas), we write $E \seq
E'$ iff $E$ and $E'$ are syntactically identical up to a renaming of
bound variables.  We say that a term $t$ is \emph{free for $x$ in $E$} iff
$x$ does not occur free in the scope of an \eps-operator $\meps y{}$ or
quantifier $\forall y$, $\exists y$ for any~$y \in \FV(t)$.

If $E$ is an expression and $t$ is a term, we write $\st E x t$ for
the result of substituting every free occurrence of~$x$ in~$E$ by~$t$,
provided $t$ is free for $x$ in $E$, and renaming bound variables
in~$t$ if necessary.  We write $E(x)$ to indicate that $x \in \FV(E)$,
and $E(t)$ for $\st E x t$.  We write $\ST E t u$ for the result of
replacing every occurrence of $t$ in $E$ by $u$.\footnote{Skipping
  details, (a)~we want to replace not just every occurrence of $t$ by
  $u$, but every occurrence of a term $t' \seq t$.  (b)~$t$ may have
  an occurrence in $E$ where a variable in $t$ is bound by a
  quantifier or \eps{} outside $t$, and such occurrences shouldn't be
  replaced (they are not subterm occurrences). (c)~When replacing $t$
  by $u$, bound variables in $u$ might have to be renamed to avoid
  conflicts with the bound variables in $E'$ and bound variables in
  $E'$ might have to be renamed to avoid free variables in $u$ being
  bound.}

\begin{defn}[\eps-Translation]
  If $E$ is an expression, define $E^\eps$ by:
  \begin{enumerate}
  \item $E^\eps = E$ if $E$ is a variable, a constant symbol,
    or~$\bot$.
  \item If $E = f^n_i(t_1, \dots, t_n)$, $E^\eps = f^n_i(t_1^\eps,
    \dots, t_n^\eps)$.
  \item If $E = P^n_i(t_1, \dots, t_n)$, $E^\eps = P^n_i(t_1^\eps,
    \dots, t_n^\eps)$.
  \item If $E = \lnot A$, then $E^\eps = \lnot A^\eps$.
  \item If $E = (A \land B)$, $(A \lor B)$, $(A \lif B)$, or $(A \liff
    B)$, then $E^\eps = (A^\eps \land B^\eps)$, $(A^\eps \lor
    B^\eps)$, $(A^\eps \lif B^\eps)$, or $(A^\eps \liff B^\eps)$,
    respectively.
  \item  If $E = \exists x\, A(x)$ or $\forall x\, A(x)$, then $E^\eps
    = A^\eps(\meps x{A(x)^\eps})$ or $A^\eps(\meps x {\lnot A(x)^\eps})$.
  \item If $E = \meps x {A(x)}$, then $E^\eps = \meps x{A(x)^\eps}$.
  \end{enumerate}
\end{defn}

\begin{defn}
  An \eps-term $p \seq \meps x {B(x; x_1, \dots, x_n)}$ is a
  \emph{type of an \eps-term~$\meps x {A(x)}$} iff
  \begin{enumerate}
  \item $p \seq \st{\st{\meps x {A(x)}}{x_1}{t_1}\dots}{x_n}{t_n}$
  for some terms $t_1$, \dots,~$t_n$.
  \item $\FV(p) = \{x_1, \dots, x_n\}$.
  \item $x_1$, \dots, $x_n$ are all immediate subterms of $p$.
  \item Each $x_i$ has exactly one occurrence in~$p$.
  \item The occurrence of $x_i$ is left of the occurrence of $x_j$ in
    $p$ if $i < j$.
  \end{enumerate}
  We denote the set of types as~\Typ.
\end{defn}

\begin{proposition}
  The type of an epsilon term~$\meps x {A(x)}$ is unique up to
  renaming of bound, and disjoint renaming of free variables.
\end{proposition}

\begin{defn}
  An \eps-term $e$ is \emph{nested in} an \eps-term $e'$ if $e$ is a
  proper subterm of~$e$.
\end{defn}

\begin{defn}
  The \emph{degree~$\deg e$} of an \eps-term~$e$ is defined as
  follows: (1) $\deg e = 1$ iff $e$ contains no nested \eps-terms.
  (2) $\deg e = \max\{\deg {e_1}, \dots, \deg{e_n}\} + 1$ if $e_1$,
  \dots,~$e_n$ are all the \eps-terms nested in~$e$.  For convenience,
  let $\deg{t} = 0$ if $t$ is not an \eps-term.
\end{defn}

\begin{defn}
  An \eps-term $e$ is \emph{subordinate to} an \eps-term $e' = \meps x
  A(x)$ if some $e'' \seq e$ occurs in $e'$ and $x \in \FV(e'')$.
\end{defn}

Note that if $e$ is subordinate to $e'$ it is \emph{not} a subterm of
$e'$, because $x$ is free in $e$ and so the occurrence of $e$ (really,
of the variant $e''$) in $e'$ is in the scope of~$\eps_x$.\footnote{One might
think that replacing $e$ in $\meps x{A(x)}$ by a new variable $y$
would result in an \eps-term $\meps x{A'(y)}$ so that $e' \equiv \st
{\meps x{A'(y)}}{y}{e}$. But (a) $\meps x{A'(y)}$ is not in general a
term, since it is not guaranteed that $x$ is free in $A'(y)$ and (b)
$e$ is not free for $y$ in $\meps x {A'(y)}$.}

\begin{defn}
  The \emph{rank~$\rk e$} of an \eps-term~$e$ is defined as follows:
  (1) $\rk e = 1$ iff $e$ contains no subordinate \eps-terms. (2) $\rk
  e = \max\{\rk {e_1}, \dots, \rk{e_n}\} + 1$ if $e_1$, \dots,~$e_n$
  are all the \eps-terms subordinate to~$e$.
\end{defn}

\begin{proposition}
  If $p$ is the type of $e$, then $\rk{p} = \rk{e}$.
\end{proposition}

\subsection{Axioms and Proofs}

\begin{defn} The axioms of the \emph{elementary calculus}~\EC{} are
\begin{align}
  & A & \text{for any tautology~$A$} \tag{Taut}
\end{align}
and its only rule of inference is
\[\infer[$MP$]{A}{A & A \lif B}\]
For \ECq, we add
\begin{align}
t & = t & \text{for any term~$t$} \tag{$=_1$} \\
t = u & \lif (\st A x t \liff \st A x u). \tag{$=_2$}
\end{align}
The axioms and rules of the (intensional) \emph{\eps-calculus}~\ECe{} (\ECeq)
are those of~\EC{} (\ECq) plus the \emph{critical formulas}
\begin{align}
A(t) \lif A(\meps x {A(x)}). \tag{crit}
\end{align}
The axioms and rules of the \emph{extensional \eps-calculus}~$\ECex$
are those of~\ECeq{} plus
\begin{align}
  (\forall x(A(x) \liff B(x)))^\eps & \lif \meps x {A(x)} = \meps x
  {B(x)}, \tag{ext} \\
  \intertext{that is,}
  A(\meps x{\lnot(A(x) \liff B(x))}) \liff B(\meps x{\lnot(A(x) \liff
    B(x))}) & \lif \meps x {A(x)} = \meps x {B(x)} \notag
\end{align}
The axioms and rules of \PC, \PCe, $\PCe^\ext$ are those of \EC, \ECe,
$\ECe^\ext$, respectively, together with the axioms
\begin{align}
A(t) & \lif \exists x\, A(x)  \tag{Ax$\exists$} \\
\forall x\, A(x) & \lif A(t) \tag{Ax$\forall$}
\end{align}
and the rules
\[
\infer[R\exists]{\exists x\, A(x) \lif B}{A(x) \lif B} \qquad
\infer[R\forall]{B \lif \forall x\, A(x)}{B \lif A(x)} 
\]
Applications of these rules must satisfy the \emph{eigenvariable
  condition}, viz., the variable $x$ must not appear in the conclusion
or anywhere below it in the proof.
\end{defn}

\begin{defn}
  If $\Gamma$ is a set of formulas, a \emph{proof of $A$ from $\Gamma$
    in $\PCe^\ext$} is a sequence~$\pi$ of formulas $A_1$, \dots, $A_n
  = A$ where for each $i \le n$, $A_i \in \Gamma$, $A_i$ is an
  instance of an axiom, or follows from formulas $A_j$ ($j<i$) by a
  rule of inference.
  
  If $\pi$ only uses the axioms and rules of \EC, \ECe, $\ECe^\ext$,
  etc., then it is a proof of $A$ from $\Gamma$ in \EC, \ECe,
  $\ECe^\ext$, etc., and we write $\Gamma \proves[\pi]{} A$, $\Gamma
  \proves[\pi]{\eps} A$, $\Gamma \proves[\pi]{\eps\ext} A$, etc. 

  We say that $A$ is provable from $\Gamma$ in \EC, etc. ($\Gamma
  \proves{} A$, etc.), if there is a proof of $A$ from $\Gamma$ in
  \EC, etc.
\end{defn}

Note that our definition of proof, because of its use of $\seq$,
includes a tacit rule for renaming bound variables. Note also that
substitution into members of~$\Gamma$ is \emph{not} permitted.
However, we can simulate a provability relation in which substitution
into members of $\Gamma$ is allowed by considering $\inst \Gamma$, the
set of all substitution instances of members of~$\Gamma$.  If $\Gamma$
is a set of sentences, then $\inst\Gamma = \Gamma$.

\begin{proposition}\label{proof-subst}
  If $\pi = A_1$, \dots, $A_n \equiv A$ is a proof of $A$ from
  $\Gamma$ and $x \notin \FV(\Gamma)$ is not an eigenvariable in
  $\pi$, then $\st \pi x t = \st {A_1} x t$, \dots, $\st {A_n} x t$ is
  a proof of $\st A x t$ from $\inst \Gamma$.
\end{proposition}

\begin{lemma}\label{ded-lemma}
  If $\pi$ is a proof of $B$ from $\Gamma \cup \{A\}$, then there is a
  proof $\pi[A]$ of $A \lif B$ from $\Gamma$, provided $A$ 
  contains no eigenvariables of~$\pi$ free.
\end{lemma}

\begin{proof}
  By induction on the length of~$\pi$, as in the classical case.
\end{proof}

\begin{theorem}[Deduction Theorem]\label{deduction-thm}
  If $\Sigma \cup \{A\}$ is a set of sentences, $\Sigma \proves{} A \lif B$
  iff $\Sigma \cup \{A\} \proves{} B$.
\end{theorem}

\begin{corollary}\label{incons}
  If $\Sigma \cup \{A\}$ is a set of sentences, $\Sigma \proves{} A$
  iff $\Sigma \cup \{\lnot A\} \proves{} \bot$.
\end{corollary}

\begin{lemma}[\eps-Embedding Lemma]
  If $\Gamma \proves[\pi]{\eps\forall} A$, then there is a proof
  $\pi^\eps$ so that $\inst{{\Gamma^\eps}} \proves[\pi^\eps]{\eps}
  A^\eps$
\end{lemma}

\begin{proof}
By induction, see \cite{MoserZach:06}.
\end{proof}

\section{Semantics and Completeness}

\subsection{Semantics for \ECex}

\begin{defn}
  A \emph{structure}~$\M = \langle \card \M, (\cdot)^\M\rangle$
  consists of a nonempty \emph{domain}~$\card \M \neq \emptyset$ and a
  mapping $(\cdot)^\M$ on function and predicate symbols where
$(f^0_i)^\M  \in \card \M$, $(f^n_i)^M \in \card \M^{\card \M^n}$, and 
  $(P^n_i)^\M \subseteq \card \M^n$.
\end{defn}

\begin{defn}
  An \emph{extensional choice function~$\Phi$ on $\M$} is a function
  $\Phi\colon \wp(\card\M) \to \card\M$ where $\Phi(X) \in X$ whenever
  $X \neq \emptyset$.
\end{defn}

Note that $\Phi$ is total on $\wp(\card\M)$, and so
$\Phi(\emptyset) \in \card{\M}$.

\begin{defn}
  An \emph{assignment~$s$ on $\M$} is a function $s\colon \Var \to
  \card\M$.

  If $x \in \Var$ and $m \in \card\M$, $\st s x m$ is the assignment
  defined by
\[
\st s x m(y) = \begin{cases} m & \text{if $y = x$} \\ s(y) &
  \text{otherwise}
\end{cases}
\] 
\end{defn}

\begin{defn}\label{ext-sat}
  The \emph{value~$\val \M \Phi s t$ of a term} and the
  \emph{satisfaction relation $\sat \M \Phi s A$} are defined as
  follows:
\begin{enumerate}
\item $\val \M \Phi s x =  s(x)$
\item $\sat \M \Phi s \top$ and $\M, \Phi, s \not\models \bot$
\item $\val \M \Phi s {f^n_i(t_1, \dots, t_n)} = (f^n_i)^\M(\val \M
  \Phi s {t_1}, \dots, \val \M \Phi s {t_n})$ 
\item $\sat \M \Phi s {t_1 = t_n}$ iff $\val \M
  \Phi s {t_1} = \val \M \Phi s {t_2}$
\item $\sat \M \Phi s {P^n_i(t_1, \dots, t_n)}$ iff $\langle\val \M
  \Phi s {t_1}, \dots, \val \M \Phi s {t_n}\rangle \in (P^n_i)^\M$
\item\label{epsilon-sat} $\val \M \Phi s {\meps x{A(x)}} = \Phi(\val \M
  \Phi s {A(x)})$ where
\[
\val \M \Phi s {A(x)} = \{ m \in \card\M : \sat \M \Phi {\st s x m} A(x)\}
\]
\item $\sat \M \Phi s {\exists x\, A(x)}$ iff for some $m \in
  \card\M$, $\sat \M \Phi {\st s x m} {A(x)}$
\item $\sat \M \Phi s {\forall x\, A(x)}$ iff for all $m \in
  \card\M$, $\sat \M \Phi {\st s x m} {A(x)}$
\end{enumerate}
\end{defn}

\begin{proposition}
  If $s(x) = s'(x)$ for all $x \notin \FV(t) \cup \FV(A)$, then $\val
  \M \Phi s t = \val \M \Phi {s'} t$ and $\sat \M \Phi s A$ iff $\sat
  \M \Phi {s'} A$.
\end{proposition}

\begin{proposition}[Substitution Lemma]
If $m = \val \M \Phi s u$, then 
  $\val \M \Phi s {t(u)} = \val \M \Phi {\st s x m} {t(x)}$ and
  $\sat \M \Phi s {A(u)}$ iff $\sat \M \Phi {\st s x m} {A(x)}$
\end{proposition}

\begin{defn}
\begin{enumerate}
\item $A$ is \emph{locally true} in $\M$ w.r.t.\ $\Phi$ and~$s$ iff
  $\sat \M \Phi s A$.
\item $A$ is \emph{true} in $\M$ with respect to $\Phi$, $\M, \Phi
  \models A$, iff for all $s$ on $\M$: $\sat \M \Phi s A$.
\item $A$ is \emph{generically true} in $\M$ with respect to~$s$, $\M,
  s \models^g A$, iff for all choice functions $\Phi$
  on~$\M$: $\sat \M \Phi s A$.
\item $A$ is \emph{generically valid} in $\M$, $\M \models A$, if for
  all choice functions $\Phi$ and assignments~$s$ on~$\M$: $\sat \M
  \Phi s A$.
\end{enumerate}
\end{defn}

\begin{defn} Let $\Gamma \cup\{A\}$ be a set of formulas.
\begin{enumerate}
\item $A$ is a \emph{local consequence} of $\Gamma$, $\Gamma
  \models^l A$, iff for all $\M$, $\Phi$, and $s$:\\
  \qquad if $\sat \M \Phi s \Gamma$ then $\sat \M \Phi s A$.
\item $A$ is a \emph{truth consequence} of $\Gamma$, $\Gamma
  \models A$, iff for all $\M$, $\Phi$:\\
  \qquad if $\M, \Phi \models \Gamma$
  then $\M, \Phi \models A$.
\item $A$ is a \emph{generic consequence} of $\Gamma$, $\Gamma
  \models^g A$, iff for all $\M$ and $s$:\\
  \qquad if $\M, s \models^g \Gamma$
  then $\M \models A$.
\item $A$ is a \emph{generic validity consequence} of $\Gamma$, $\Gamma
  \models^v A$, iff for all $\M$:\\
  \qquad if $\M \models^v \Gamma$ then $\M
  \models A$.
\end{enumerate}
\end{defn}

\begin{proposition}
  If $\Sigma \cup \{A\}$ is a set of sentences, $\Sigma \models^l A$
  iff $\Sigma \models A$
\end{proposition}

\begin{proposition}
  If $\Sigma \cup \{A, B\}$ is a set of sentences, $\Sigma \cup \{A\}
  \models B$ iff $\Sigma \models A \lif B$.
\end{proposition}

\begin{corollary}
  If $\Sigma \cup \{A\}$ is a set of sentences, $\Sigma \models A$
  iff for no $\M$, $\Phi$, $\M \models \Sigma \cup \{\lnot A\}$
\end{corollary}

\subsection{Soundness and Completeness}

\begin{theorem}
If $\Gamma \proves{\eps} A$, then $\Gamma \models^l A$.
\end{theorem}

\begin{proof}
  Suppose $\sat\Gamma\Phi s \Gamma$.  We show by induction on the
  length $n$ of a proof $\pi$ that $\sat \M \Phi {s'} A$ for all $s'$
  which agree with $s$ on $\FV(\Gamma)$.  We may assume that no
  eigenvariable~$x$ of $\pi$ is in $\FV(\Gamma)$ (if it is, let $y
  \notin \FV(\pi)$ and not occurring in~$\pi$; consider $\st \pi x y$
  instead of $\pi$).

  If $n = 0$ there's nothing to prove. Otherwise, we distinguish cases
  according to the last line $A_n$ in $\pi$. The only interesting case is when
  $A_n$ is a critical formula, i.e., $A_n \equiv A(t) \lif A(\meps
  x {A(x)})$. Then either $\sat \M \Phi s {A(t)}$ or not (in which
  case there's nothing to prove). If yes, $\sat\M \Phi {s[x/m]} {A(x)}$
  for $m = \val \M \Phi s {t}$, and so $Y = \val \M \Phi s {A(x)} \neq
  \emptyset$.  Consequently, $\Phi(Y) \in Y$, and hence $\sat \M \Phi
  s {A(\meps x {A(x)})}$.
\end{proof}

\begin{lemma}
  If $\Gamma$ is a set of sentences and $\Gamma
  \nproves{\eps} \bot$, then there are $\M$, $\Phi$ so that $\M,
  \Phi \models \Gamma$.
\end{lemma}

\begin{theorem}[Completeness]
  If $\Gamma \cup \{A\}$ are sentences and $\Gamma \models
  A$, then $\Gamma \proves{\eps\ext} A$.
\end{theorem}

\begin{proof}
  Suppose $\Gamma \not\models A$. Then for some $\M$, $\Phi$ we have
  $\M, \Phi \models \Gamma$ but $\M, \Phi \not\models A$. Hence $\M,
  \Phi \models \Gamma \cup \{\lnot A\}$.  By the Lemma, $\Gamma \cup
  \{\lnot A\} \proves{\eps} \bot$. By Corollary~\ref{incons}, $\Gamma
  \proves{\eps} A$.
\end{proof}

The proof of the Lemma comes in several stages.  We have to show that
if $\Gamma$ is consistent, we can construct $\M$, $\Phi$, and $s$ so
that $\sat \M \Phi s \Gamma$. Since $\FV(\Gamma) = \emptyset$, we then
have $\M, \Phi \models \Gamma$.

\begin{lemma}
  If $\Gamma \nproves{\eps} \bot$, there is $\Gamma^* \supseteq
  \Gamma$ with (1) $\Gamma^* \nproves{\eps} \bot$ and (2) for all
  formulas $A$, either $A \in \Gamma^*$ or $\lnot A \in \Gamma^*$.
\end{lemma}

\begin{proof}
  Let $A_1$, $A_2$, \dots\ be an enumeration of $\Frm_\eps$. Define
  $\Gamma_0 = \Gamma$ and \[ \Gamma_{n+1} =
\begin{cases}
  \Gamma_n \cup \{A_n\} & 
  \text{if $\Gamma_n \cup \{A_n\} \nproves{\eps} \bot$} \\
  \Gamma_n \cup \{\lnot A_n\} & 
  \text{if $\Gamma_n \cup \{\lnot A_n\}
    \nproves{\eps} \bot$ otherwise}
\end{cases}
\]
Let $\Gamma^* = \bigcup_{n\ge 0} \Gamma_n$. Obviously, $\Gamma
\subseteq \Gamma^*$. For (1), observe that if $\Gamma^*
\proves[\pi]{\eps} \bot$, then $\pi$ contains only finitely many
formulas from~$\Gamma^*$, so for some $n$, $\Gamma_n
\proves[\pi]{\eps} \bot$. But $\Gamma_n$ is consistent by definition.

To verify (2), we have to show that for each $n$, either $\Gamma_n
\cup \{A_n\} \nproves{\eps} \bot$ or $\Gamma_n \cup \{\lnot A\}
\nproves{\eps} \bot$. For $n = 0$, this is the assumption of the lemma.
So suppose the claim holds for $n-1$.  Suppose $\Gamma_n \cup \{A\}
\proves[\pi]{\eps} \bot$ and $\Gamma_n \cup \{\lnot A\}
\proves[\pi']{\eps} \bot$. Then by the Deduction Theorem, we have
$\Gamma_n \proves[{\pi[A]}] A \lif \bot$ and $\Gamma_n
\proves[{\pi'[A']}] \lnot A \lif \bot$. Since $(A \lif \bot) \lif
((\lnot A \lif \bot) \lif \bot)$ is a tautology, we have $\Gamma_n
\proves{\eps} \bot$, contradicting the induction hypothesis.
\end{proof}

\begin{lemma}\label{lem-closed}
If $\Gamma^* \proves{\eps} B$, then $B \in \Gamma^*$.
\end{lemma}

\begin{proof}
  If not, then $\lnot B \in \Gamma^*$ by maximality, so $\Gamma^*$
  would be inconsistent.
\end{proof}

\begin{defn}
  Let $\approx$ be the relation on $\Trm_\eps$ defined by 
\[ 
t \approx u \text{ iff } t = u \in \Gamma^*
\]
It is easily seen that $\approx$ is an equivalence relation.  Let
$\widetilde t = \{u : u \approx t\}$ and $\widetilde \Trm =
\{\widetilde t : t \in \Trm\}$.
\end{defn}

\begin{defn}
A set $T \in \widetilde \Trm$ is \emph{represented by $A(x)$}
if $T = \{\widetilde t : A(t) \in \Gamma^*\}$.

Let $\Phi_0$ be a fixed choice function on $\widetilde \Trm$, and define
\[
\Phi(T) = \begin{cases}
\widetilde{\meps x {A(x)}} & \text{if $T$ is represented by $A(x)$}\\
\Phi_0(T) & \text{otherwise.}
\end{cases}
\]
\end{defn}

\begin{proposition}
$\Phi$ is a well-defined choice function on $\widetilde \Trm$.
\end{proposition}

\begin{proof}
  Use (ext) for well-definedness and (crit) for choice function.
\end{proof}

Now let $\M =\langle \widetilde \Trm, (\cdot)^\M\rangle$ with $c^\M =
\widetilde c$, $(P_i^n)^\M = \{\langle\widetilde t_1, \dots,
\widetilde t_1\rangle : P_i^n(t_1, \ldots, t_n)\}$, and let $s(x) =
\widetilde s$.

\begin{proposition}
  $\sat \M \Phi s {\Gamma^*}$.
\end{proposition}

\begin{proof}
  We show that $\val \M \Phi s t = \widetilde t$ and $\sat \M \Phi s
  A$ iff $A \in \Gamma^*$ by simultaneous induction on the complexity
  of $t$ and $A$.

  If $t = c$ is a constant, the claim holds by definition of
  $(\cdot)^\M$.  If $A = \bot$ or $= \top$, the claim holds by
  Lemma~\ref{lem-closed}.

  If $A \equiv P^n(t_1, \ldots, t_n)$, then by induction hypothesis,
  $\val \M \Phi s t_i = \widetilde {t_i}$.  By definition of
  $(\cdot)^\M$, $\langle \widetilde{t_1}, \dots,
  \widetilde{t_n}\rangle \in (P^n_i)(t_1, \dots, t_n)$ iff $P^n_i(t_1,
  \dots, t_n) \in \Gamma^*$.

  If $A\equiv \lnot B$, $(B \land C)$, $(B \lor C)$, $(B \lif C)$, $(B
  \liff C)$, the claim follows immediately from the induction
  hypothesis and the definition of $\models$ and the closure
  properties of $\Gamma^*$. For instance, $\sat \M \Phi s {(B \land
    C)}$ iff $\sat \M \Phi s B$ and $\sat \M \Phi s C$. By induction
  hypothesis, this is the case iff $B \in \Gamma^*$ and $C \in
  \Gamma^*$. But since $B, C \proves{\eps} B \land C$ and $B \land C
  \proves{\eps} B$ and $\proves{\eps} C$, this is the case iff $(B
  \land C) \in \Gamma^*$.  Remaining cases: Exercise.

  If $t \seq \meps x {A(x)}$, then $\val \M \Phi s t = \Phi(\val \M
  \Phi s {A(x)})$. Since $\val \M \Phi s {A(x)} $ is represented by
  $A(x)$ by induction hypothesis, we have $\val \M \Phi s t =
  \widetilde{\meps x {A(x)}}$ by definition of $\Phi$.
\end{proof}

\subsection{Semantics for \ECe}

In order to give a complete semantics for \ECe, i.e., for the calculus
without the extensionality axiom~(ext), it is necessary to change the
notion of choice function so that two \eps-terms \meps x {A(x)} and
\meps x {B(x)} may be assigned different representatives even when
$\sat \M \Phi s {\forall x(A(x) \liff B(x))}$, since then the negation
of (ext) is consistent in the resulting calculus.  The idea is to add
the \eps-term itself as an additional argument to the choice function.
However, in order for this semantics to be sound for the
calculus---specifically, in order for ($=_2$) to be valid---we have to
use not \eps-terms but \eps-types.

\begin{defn}
  An \emph{intensional choice operator} is a mapping $\Psi\colon \Typ
  \times \card{\M}^{<\omega} \to \card{\M}^{\wp(\card{\M})}$ such that
  for every type $p = \meps x{A(x; y_1, \dots, y_n)}$ is a type, and
  $m_1$, \dots,~$m_n \in \card{\M}$, $\Psi(p, m_1, \dots, m_n)$ is a
  choice function.
\end{defn}

\begin{defn}
  If $\M$ is a structure, $\Psi$ an intensional choice operator, and
  $s$ an assignment, $\val \M \Psi s t$ and $\sat \M \Psi s A$ is
  defined as before, except (\ref{epsilon-sat}) in Definition~\ref{ext-sat} is
  replaced by:
  \begin{enumerate}
  \item[($\ref{epsilon-sat}'$)] $\val \M \Psi s {\meps x{A(x)}} =
    \Psi(p, m_1, \dots, m_n)(\val \M \Phi s {A(x)})$ where 
    \begin{enumerate}
    \item $p = \meps x{A'(x; x_1, \dots, x_n)}$ is the type of $\meps
      x {A(x)}$,
    \item $t_1$, \dots, $t_n$ are the subterms corresponding to $x_1$,
      \dots, $x_n$, i.e., $\meps x{A(x)} \seq \meps x{A'(x; t_1,
        \dots, t_n)}$, 
    \item $m_i = \val \M \Psi s t_1$, and 
    \item
    $\val \M \Phi s {A(x)} = 
    \{ m \in \card\M : \sat \M \Psi {\st s x m} A(x)\}$
    \end{enumerate}
\end{enumerate}
\end{defn}

The soundness and completeness proofs generalize to \ECe, \ECeq,
and~\PCe.

\section{The First Epsilon Theorem}\label{epsthm}

\subsection{The Case Without Identity}

\begin{defn}
  An \eps-term $e$ is \emph{critical in~$\pi$} if $A(t) \lif A(e)$ is
  one of the critical formulas in $\pi$.  The \emph{rank~$\rk{\pi}$ of
    a proof~$\pi$} is the maximal rank of its critical \eps-terms. The
  \emph{$r$-degree~$\deg{\pi, r}$} of $\pi$ is the maximum degree of
  its critical \eps-terms of rank~$r$. The \emph{$r$-order~$o(\pi,
    r)$} of $\pi$ is the number of different (up to renaming of bound
  variables) critical \eps-terms of rank~$r$.
\end{defn}

\begin{lemma}
  If $e = \meps x{A(x)}$, $\meps y {B(y)}$ are critical in $\pi$,
  $\rk{e} = \rk{\pi}$, and $B^* \equiv B(u) \lif B(\meps y{B(y)})$ is a
  critical formula in~$\pi$. Then, if $e$ is a subterm of $B^*$, it is
  a subterm of $B(y)$ or a subterm of~$u$.
\end{lemma}

\begin{proof}
  Suppose not. Since $e$ is a subterm of $B^*$, we have $B(y)
  \seq B'(\meps x{A'(x, y)}, y)$ and either $e \seq \meps x{A'(x, u)}$
  or $e \seq \meps x{A'(x, \meps y{B(y)})}$. In each case, we see that
  $\meps x{A'(x, y)}$ and $e$ have the same rank, since the latter is
    an instance of the former (and so have the same type). On the
    other hand, in either case, $\meps y {B(y)}$ would be
  \[
  \meps y {B'(\meps x{A'(x, y)}, y)}
  \]
  and so would have a higher rank than $\meps x{A'(x, y)}$ as that
  \eps-term is subordinate to it.  This contradicts $\rk{e} = \rk{\pi}$.
\end{proof}

\begin{lemma}
  Let $e$, $B^*$ be as in the lemma, and $t$ be any term. Then 
  \begin{enumerate}
  \item If $e$ is not a subterm of $B(y)$, $\ST{B^*}{e}{t} \equiv
    B(u') \lif B(\meps y{B(y)})$.
  \item If $e$ is a subterm of $B(y)$, i.e., $B(y) \seq B'(e, y)$,
    $\ST{B^*}{e}{t} \equiv B'(t, u') \lif B'(t, \meps{y}{B'(t, y)})$.
  \end{enumerate}
\end{lemma}

\begin{lemma}
  If $\proves[\pi]{\eps} E$ and $E$ does not contain \eps, then there
  is a proof $\pi'$ such that $\proves[\pi']{\eps} E$ and $\rk{\pi'}
  \le \rk{pi} = r$ and $o(\pi', r) < o(\pi, r)$.
\end{lemma}

\begin{proof}
  Let $e$ be an \eps-term critical in $\pi$ and let $A(t_1) \lif
  A(e)$, dots, $A(t_n) \lif A(e)$ be all its critical formulas in
  $\pi$.

  Consider $\ST \pi e t_i$, i.e., $\pi$ with $e$ replaced by $t_i$
  throughout.  Each critical formula belonging to $e$ now is of the
  form $A(t_j') \lif A(t_i)$, since $e$ obviously cannot be a subterm
  of $A(x)$ (if it were, $e$ would be a subterm of $\meps x {A(x)}$,
  i.e., of itself!). Let $\hat\pi_i$ be the sequence of tautologies
  $A(t_i) \lif (A(t_j') \lif A(t_i))$ for $i = 1$, \dots,~$n$,
  followed by $\ST \pi e t_i$.  Each one of the formulas $A(t_j') \lif
  A(t_i)$ follows from one of these by (MP) from $A(t_i)$. Hence,
  $A(t_i) \proves[\hat\pi_i]{\eps} E$. Let $\pi_i = \hat\pi_i[A_i]$ as
  in Lemma~\ref{ded-lemma}.  We have $\proves[\pi_i]{\eps} A_i \lif E$.

  The \eps-term $e$ is not critical in $\pi_i$: Its original critical
  formulas are replaced by $A(t_i) \lif (A(t_j') \lif A(t_i))$, which
  are tautologies.  By (1) of the preceding Lemma, no critical
  \eps-term of rank~$r$ was changed at all. By (2) of the preceding
  Lemma, no critical \eps-term of rank $< r$ was replaced by a
  critical \eps-term of rank~$\ge r$. Hence, $o(\pi_i, r) = o(\pi) -
  1$.

  Let $\pi''$ be the sequence of tautologies $\lnot \bigvee_{i=1}^n
  A(t_i) \lif (A(t_i) \lif A(e))$ followed by $\pi$.  Then
  $\bigvee_{i=1}^n A(t_i) \proves[\pi''] E$, $e$ is not critical in
  $\pi''$, and otherwise $\pi$ and $\pi''$ have the same critical
  formulas. The same goes for $\pi''[\lnot \bigvee A(t_i)]$, a proof
  of $\lnot\bigvee A(t_i) \lif E$.  

  We now obtain $\pi'$ as the $\pi_i$, $i = 1$, \dots,~$n$, followed
  by $\pi[\lnot \bigvee_{i=1}^n A(t_i)]$, followed by the tautology
  \[
  (\lnot \bigvee A(t_i) \lif E) \lif (A(t_1) \lif E) \lif \dots \lif
  (A(t_n) \lif E) \lif E)\dots)
  \]
  from which $E$ follows by $n+1$ applications of (MP).
\end{proof}

\begin{theorem}[First Epsilon Theorem for \ECe]
  If $E$ is a formula not containing any \eps-terms
  and $\proves{\eps} E$, then $\proves{\eps} E$.
\end{theorem}

\begin{proof}
  By induction on $o(\pi, r)$, we have: if $\proves[\pi]{\eps} E$,
  then there is a proof $\pi^*$ of $E$ with $\rk{\pi^-} < r$. By
  induction on $\rk(\pi)$ we have a proof $\pi^{**}$ of $E$ with
  $\rk{\pi^{**}} = 0$, i.e., without critical formulas at all.
\end{proof}

\begin{corollary}[Extended First \eps-Theorem]
  If $\proves{\eps} E(e_1, \dots, e_n)$, then $\proves
  \bigvee_{i=1}^m E(t_1^j, \dots, t_n^j)$ for some terms $t_j$ (in \EC).
\end{corollary}

\begin{proof}
If $E$ contains \eps-terms, say, $E$ is $E(e_1, \dots, e_n)$, then
replacement of \eps-terms in the construction of $\pi_i$ may change
$E$---but of course only the \eps-terms appearing as subterms in
it. In each step we obtain not a proof of $E$ but of some disjunction
of instances $E(e_1', \dots, e_n')$. For details, see
\cite{MoserZach:06}.
\end{proof}

\subsection{The Case with Identity}

In the presence of the identity ($=$) predicate in the language,
things get a bit more complicated. The reason is that instances of the
($=_2$) axiom schema,
\[
t = u \lif (A(t) \lif A(u))
\]
may also contain \eps-terms, and the replacement of an \eps-term~$e$
by a term~$t_i$ in the construction of $\pi_i$ may result in a formula
which no longer is an instance of ($=_2$). For instance, suppose that
$t$ is a subterm of $e = e'(t)$ and $A(t)$ is of the form $A'(e'(t))$.
Then the original axiom is
\[
t = u \lif (A'(e'(t)) \lif A'(e'(u))
\] 
which after replacing $e = e'(t)$ by $t_i$ turns into
\[
t = u \lif (A'(t_i) \lif A'(e'(u)).
\] 
So this must be avoided.  In order to do this, we first observe that
just as in the case of the predicate calculus, the instances of
($=_2$) can be derived from restricted instances. In the case of the
predicate calculus, the restricted axioms are 
\begin{align}
  t = u & \lif (P^n(s_1, \dots, t, \dots s_n) \lif P^n(s_1, \dots, u,
  \dots, s_n) \tag{$=_2'$} \\
  t = u & \lif f^n(s_1, \dots, t, \dots, s_n) = f^n(s_1, \dots, u,
  \dots, s_n) \tag{$=_2''$} \\
\intertext{to which we have to add the \emph{\eps-identity axiom schema:}}
  t = u & \lif \meps x {A(x; s_1, \dots, t, \dots s_n)} = 
  \meps x {A(x; s_1, \dots, u, \dots s_n)} \tag{$=_\eps$}
\end{align}
where $\meps x {A(x; x_1, \dots, x_n)}$ is an \eps-type.

\begin{proposition}\label{atomic}
  Every instance of $(=_2)$ can be derived from $(=_2')$, $(=_2'')$,
  and $(=_\eps)$.
\end{proposition}

\begin{proof}
  By induction.
\end{proof}

Now replacing every occurrence of $e$ in an instance of ($=_2'$) or
($=_2''$)---where $e$ obviously can only occur inside one of the terms
$t$, $u$, $s_1$, \dots, $s_n$---results in a (different) instance of
($=_2'$) or ($=_2''$).  The same is true of ($=_\eps$), \emph{provided
  that} the $e$ is neither $\meps x{A(x; s_1, \dots, t, \dots s_n)}$
nor $\meps x{A(x; s_1, \dots, u, \dots s_n)}$.  This would be
guaranteed if the type of $e$ is not $\meps x {A(x; x_1, \dots,
  x_n)}$, in particular, if the rank of $e$ is higher than the rank of
$\meps x {A(x; x_1, \dots, x_n)}$.  Moreover, the result of replacing
$e$ by $t_i$ in any such instance of $(=_\eps$) results in an instance
of $(=_\eps)$ which belongs to the same \eps-type. Thus, in order for
the proof of the first \eps-theorem to work also when $=$ and axioms
$(=_1)$, $(=_2')$, $(=_2''$), and $(=_\eps)$ are present, it suffices
to show that the instances of $(=_\eps)$ with \eps-terms of
rank~$\rk{\pi}$ can be removed.  Call an \eps-term $e$ \emph{special}
in~$\pi$, if $\pi$ contains an occurrence of $t = u \lif e' = e$ as an
instance of $(=_\eps)$.

\begin{theorem}
  If $\proves[\pi]{\eps=} E$, then there is a proof $\pi^=$ so that
  $\proves[\pi^=]{\eps=} E$, $\rk{\pi^=} = \rk{pi}$, and the rank of the
  special \eps-terms in $\pi^=$ has rank $< \rk{\pi}$.
\end{theorem}

\begin{proof}
  The basic idea is simple: Suppose $t = u \lif e' = e$ is an instance
of $(=_\eps)$, with $e' \seq \meps x {A(x; s_1, \dots, t, \dots s_n)}$
and $e \seq \meps x {A(x; s_1, \dots, u, \dots s_n)}$.  Replace $e$
everywhere in the proof by~$e'$.  Then the instance of $(=_\eps)$
under consideration is removed, since it is now provable from $e' =
e'$.  This potentially interferes with critical formulas belonging to
$e$, but this can also be fixed: we just have to show that by a
judicious choice of~$e$ it can be done in such a way that the other
$(=_\eps)$ axioms are still of the required form.

Let $p = \meps x{A(x; x_1, \dots, x_n)}$ be an \eps-type of rank
$\rk{\pi}$, and let $e_1$, \dots, $e_l$ be all the \eps-terms of
type~$p$ which have a corresponding instance of~$(=_\eps)$ in~$\pi$.
Let $T_i$ be the set of all immediate subterms of $e_1$, \dots, $e_l$,
in the same position as $x_i$, i.e., the smallest set of terms so that
if $e_i \seq \meps x{A(x; t_1, \dots, t_n)}$, then $t_i \in T$. Now
let let $T^*$ be all instances of $p$ with terms from $T_i$
substituted for the $x_i$.  Obviously, $T$ and thus $T^*$ are finite
(up to renaming of bound variables). Pick a strict order $\prec$ on
$T$ which respects degree, i.e., if $\deg{t} < \deg{u}$ then $t \prec
u$.  Extend $\prec$ to $T^*$ by
\[
 \meps x{A(x; t_1, \dots, t_n)} \prec  \meps x{A(x; t'_1, \dots, t'_n)} 
\]
iff 
\begin{enumerate}
\item $\max\{\deg{t_i} : i = 1, \dots, n\} < \max\{\deg{t_i} : i = 1,
  \dots, n\}$ or
\item $\max\{\deg{t_i} : i = 1, \dots, n\} = \max\{\deg{t_i} : i = 1,
  \dots, n\}$  and
\begin{enumerate}
\item $t_i \seq t_i'$ for $i = 1$, \dots,~$k$.
\item $t_{k+1} \prec t_{k+1}'$
\end{enumerate}
\end{enumerate}
\end{proof}
  
\begin{lemma}
  Suppose $\proves[\pi]{\eps=} E$, $e$ a special \eps-term in $\pi$
  with $\rk{e} = \rk{\pi}$, $\deg{e}$ maximal among the special
  \eps-terms of rank~$\rk{\pi}$, and $e$ maximal with respect to
  $\prec$ defined above. Let $t = u \lif e' = e$ be an instance of
  $(=_\eps)$ in $\pi$. Then there is a proof $\pi'$,
  $\proves[\pi']{\eps=} E$ such that
  \begin{enumerate}
    \item $\rk{\pi'} = \rk{\pi}$
    \item $\pi'$ does not contain $t = u \lif e' = e$ as an axiom
    \item Every special \eps-term $e''$ of $\pi'$ with the same type
      as $e$ is so that $e'' \prec e$.
    \end{enumerate}
\end{lemma}

\begin{proof}
  Let $\pi_0 = \ST \pi e {e'}$ and suppose $t' = u' \lif e''' = e''$
  is an $(=_\eps)$ axiom in $\pi$.

  If $\rk{e''} < \rk{e}$, then the replacement of $e$ by $e'$ can only
  change subterms of $e''$ and $e'''$. In this case, the uniform
  replacement results in another instance of $(=_\eps)$ with
  \eps-terms of the same \eps-type, and hence of the same rank $<
  \rk{\pi}$, as the original.

  If $\rk{e''} = \rk{e}$ but has a different type than $e$, then this
  axiom is unchanged in $\pi_0$: Neither $e''$ nor $e'''$ can be $\seq
  e$, because they have different \eps-types, and neither $e''$ nor
  $e'''$ (nor $t'$ or $u'$, which are subterms of $e''$, $e'''$) can
  contain $e$ as a subterm, since then $e$ wouldn't be degree-maximal
  among the special \eps-terms of $\pi$ of rank~$\rk{\pi}$.

  If the type of $e''$, $e'''$ is the same as that of $e$, $e$ cannot
  be a proper subterm of $e''$ or $e'''$, since otherwise $e''$ or
  $e'''$ would again be a special \eps-term of rank~$\rk{\pi}$ but of
  higher degree than~$e$.  So either $e \seq e''$ or $e \seq e'''$,
  without loss of generality suppose $e \seq e''$.  Then the
  $(=_\eps)$ axiom in question has the form
  \[
  t' = u' \lif \underbrace{\meps x {A(x; s_1, \dots t', \dots
      s_n)}}_{e'''} = \underbrace{\meps x {A(x; s_1, \dots u', \dots
      s_n)}}_{e'' \seq e}
  \]
  and with $e$ replaced by $e'$:
  \[
  t' = u' \lif \underbrace{\meps x {A(x; s_1, \dots t', \dots
      s_n)}}_{e'''} = \underbrace{\meps x {A(x; s_1, \dots t, \dots
      s_n)}}_{e'}
  \]
  which is no longer an instance of $(=_\eps)$, but can be proved from
  new instances of~$(=_\eps)$.  We have to distinguish two cases
  according to whether the indicated position of $t$ and $t'$ in $e'$,
  $e'''$ is the same or not.  In the first case, $u \seq u'$, and the
  new formula
  \begin{align}
    t' = u & \lif \underbrace{\meps x {A(x; s_1, \dots t', \dots
        s_n)}}_{e'''} = \underbrace{\meps x {A(x; s_1, \dots t, \dots
        s_n)}}_{e'} \notag\\
    \intertext{can be proved from $t = u$ together with} t' = t & \lif
    \underbrace{\meps x {A(x; s_1, \dots t', \dots s_n)}}_{e'''} =
    \underbrace{\meps x {A(x; s_1, \dots t, \dots
        s_n)}}_{e'} \tag{$=_\eps$}\\
    t = u & \lif (t' = u \lif t' = t) \tag{$=_2'$}
  \end{align}
  Since $e'$ and $e'''$ already occurred in~$\pi$, by assumption $e'$, $e'''
  \prec e$.

  In the second case, the original formulas read, with terms indicated:
  \begin{align*}
    t = u & \lif \underbrace{\meps x {A(x; s_1, \dots t, \dots, u',
        \dots, s_n)}}_{e'} = \underbrace{\meps x {A(x; s_1, \dots u,
        \dots, u', \dots,
        s_n)}}_{e} \\
    t' = u' & \lif \underbrace{\meps x {A(x; s_1, \dots u, \dots, t',
        \dots, s_n)}}_{e'''} = \underbrace{\meps x {A(x; s_1, \dots u,
        \dots, u', \dots,
        s_n)}}_{e'' \seq e} 
    \intertext{and with $e$ replaced by $e'$ the latter becomes:} 
    t' = u' & \lif \underbrace{\meps x {A(x; s_1, \dots u, \dots, t', \dots
        s_n)}}_{e'''} = \underbrace{\meps x {A(x; s_1, \dots t, \dots,
        u', \dots, s_n)}}_{e'} \\
    \intertext{This new formula is provable from $t = u$ together with}
    u = t & \lif \underbrace{\meps x {A(x; s_1, \dots u, \dots, t', \dots
        s_n)}}_{e'''} = \underbrace{\meps x {A(x; s_1, \dots t, \dots,
        t', \dots, s_n)}}_{e''''} \\
    t' = u' & \lif \underbrace{\meps x {A(x; s_1, \dots t, \dots, t', \dots
        s_n)}}_{e''''} = \underbrace{\meps x {A(x; s_1, \dots t, \dots,
        u', \dots, s_n)}}_{e'}
    \end{align*}
    and some instances of $(=_2')$.  Hence, $\pi'$ contains a
    (possibly new) special \eps-term~$e''''$. However, $e''''
    \prec e$.

    In the special case where $e = e''$ and $e' = e'''$, i.e., the
    instance of $(=_\eps)$ we started with, then replacing $e$ by $e'$
    results in $t = u \lif e' = e'$, which is provable from $e' = e'$,
    an instance of $(=_1)$.

    Let $\pi_1$ be $\pi_0$ with the necessary new instances of
    $(=_\eps)$, added. The instances of $(=_\eps)$ in $\pi_1$ satisfy
    the properties required in the statement of the lemma.

    However, the results of replacing $e$ by $e'$ may have impacted
    some of the critical formulas in the original proof.  For a
    critical formula to which $e \seq \meps x {A(x, u)}$ belongs is of the form
    \begin{align}
      A(t', u) & \lif A(\meps x {A(x, u)}, u) \\
      \intertext{which after replacing $e$ by $e'$ becomes}
      A(t'', u) & \lif A(\meps x {A(x, t)}, u) \label{newcrit}
    \end{align}
    which is no longer a critical formula.  This formula, however, can
    be derived from $t = u$ together with 
    \begin{align}
      A(t'', u) & \lif A(\meps x {A(x, t)}, u) \tag{\eps}\\
      t = u & \lif (A(\meps x{A(x, t)}, t) \lif A(\meps x{A(x, t)}, u))
      \tag{$=_2$} \\
      u = t & \lif (A(t'', u) \lif A(t'', t)) \tag{$=_2$}
    \end{align}
    Let $\pi_2$ be $\pi_1$ plus these derivations of (\ref{newcrit})
    with the instances of $(=_2)$ themselves proved from $(=_2')$ and
    $(=_\eps)$.  The rank of the new critical formulas is the same, so
    the rank of $\pi_2$ is the same as that of~$\pi$.  The new
    instances of $(=_\eps)$ required for the derivation of the last
    two formulas only contain \eps-terms of lower rank that that
    of~$e$, as can be verified.

    $\pi_2$ is thus a proof of $E$ from $t = u$ which satisfies the
    conditions of the lemma.  From it, we obtain a proof $\pi_2[t =
    u]$ of $t = u \lif E$ by the deduction theorem.  On the other
    hand, the instance $t = u \lif e' = e$ under consideration can
    also be proved trivially from $t \neq u$. The proof $\pi[t \neq
    u]$ thus is also a proof, this time of $t \neq u \lif E$, which
    satisfies the conditions of the lemma. We obtain $\pi'$ by
    combining the two proofs.
\end{proof}

\begin{theorem}[First Epsilon Theorem for \ECeq]
  If $E$ is a formula not containing any \eps-terms
  and $\proves{\eps=} E$, then $\proves{=} E$ (in \ECq).
\end{theorem}

\begin{proof}
  By repeated application of the Lemma, every instance of $(=_\eps)$
  involving \eps-terms of a given type~$p$ can be eliminated from~$\pi$.
  The Theorem follows by induction on the number of different
  types of special \eps-terms of rank~$\rk{\pi}$ in~$\pi$.
\end{proof}

\section{Proof Theory of the Epsilon Calculus}\label{proofth}

\subsection{Sequent Calculi}

Leisenring \cite{Leisenring:1969} presented a one-sided sequent
calculus for the \eps-calculus. It operates on sets of formulas
(sequents); proofs are trees of sets of formulas each of which is
either an axiom (at a leaf of the tree) or follows from the sets of
formulas above it by an inference rule.  Axioms are $A, \lnot A$. The
rules are given below:
\begin{center}
\begin{tabular}{ccc}
$\infer[\land R]{\Gamma, A \land B}{\Gamma, A & \Gamma,  B}$
&
$\infer[\land L]{\Gamma, \lnot(A \land B)}{\Gamma, \lnot A, \lnot B}$
&
$\infer[\lnot\lnot]{\Gamma, \lnot\lnot A}{\Gamma, A}$
\\
$\infer[\lor R]{\Gamma, A \lor B}{\Gamma, A, B}$
&
$\infer[\lor L]{\Gamma, \lnot(A \lor B)}{\Gamma, \lnot A & \Gamma, \lnot B}$
&
$\infer[\textit{cut}]{\Pi, \Lambda}{\Pi, A & \Lambda, \lnot A}$
\\
$\infer[\exists R]{\Gamma, \exists x\, A(x)}{\Gamma, A(t)}$ 
&
$\infer[\exists L]{\Gamma, \lnot \exists x\, A(x)}{\Gamma, \lnot A(\meps x A(x))}$
&
$\infer[w]{\Gamma, A, B}{\Gamma, A}$
\\
$\infer[\forall R]{\Gamma, \forall x\, A(x)}{\Gamma, A(\meps x \lnot A(x))}$
&
$\infer[\forall L]{\Gamma, \lnot\forall x\, A(x)}{\Gamma, \lnot A(t)}$ 
\end{tabular}
\end{center}
In contrast to classical sequent systems, there are no eigenvariable
conditions!

It is complete, since proofs can easily be translated into derivations
in $\EC_\eps$; in particular it derives critical formulas:
\[
\infer[\textit{cut}]{\lnot A(t), A(\meps x A(x))}{
  \infer[\exists R]{\lnot A(t), \dem{\exists x\, A(x)}}{\lnot A(t), A(t)}
  &
  \infer[\exists L]{\dem{\lnot\exists x\, A(x)}, A(\meps x A(x))}{\lnot A(\meps x A(x)), A(\meps x A(x))}}
\]
This sequent, however, has no cut-free proof.

Maehara \cite{Maehara:55} instead proposed to simply add axioms
corresponding to to critical formulas and leave out quantifier rules.
Hence, its axioms are $\lnot A, A$ and $\lnot A(t), A(\meps x
A(x))$. It is complete, since the additional axioms allow derivation
of critical formulas.  However, it is also not cut-free complete.
Converses of critical formulas are derivable using cut:
\[
\infer[\mathit{cut}]{\lnot A(\meps x \lnot A(x)), A(t)}{
\dem{\lnot\lnot A(t)}, \lnot A(\meps x \lnot A(x)) & \dem{\lnot A(t)}, A(t)
}
\]
But these obviously have no cut-free proof.  Furthermore, addition of
these converses as axioms will not result in a cut-free complete
system, either. Consider the example given by Wessels: Let $e = \meps
x \lnot(A(x) \lor B(x))$.
\[
\infer[\mathit{cut}]{\lnot A(\meps x \lnot(A(x) \lor B(x))), A(t) \lor B(t)} 
      {
	\infer*[\mathit{cut}]{\dem{\lnot (A(e) \lor B(e))}, A(t) \lor B(t)}{
	  \dem{\lnot\lnot(A(t) \lor B(t))}, \lnot (A(e) \lor B(e))
	} 
	&
	\infer[\lor R]{\lnot A(e), \dem{A(e) \lor B(e)}}{\lnot A(e), A(e)}
      }
\]

Wessels \cite{Wessels:77} proposed to add instead the following rule to the
propositional one-sided sequent calculus:
\[
\infer[\eps0]{\Gamma, \Delta(\meps x A(x))}{
\Gamma, \Delta(z), \lnot A(z) & \Gamma, A(t) 
}
\]
Here, $\Delta(z)$ must be not empty, and $z$ may not occur in the
lower sequent.  This system also derives critical formulas, and so is
complete:
\[
\infer[\eps0]{\lnot A(t), A(\meps x A(x))}{
  \infer[w]{\underbrace{\lnot A(t)}_\Gamma, \underbrace{A(a)}_\Delta, \lnot A(z)}{A(z), \lnot A(z)} 
  &
  \underbrace{\lnot A(t)}_\Gamma, A(t)
}
\]
The rule $\eps0$ is sound.\footnote{Suppose the upper sequents are
  valid but the lower sequent is not, i.e., for some $\M, \Psi, s$,
  $\M \not\models \Gamma, \Delta(\meps x A(x))$.  In particular,
  $\M,\Psi, s \not\models \Gamma$.  Hence, $\M,\Psi,s \models A(t)$,
  i.e., $\M,\Psi,s \models A'(t, t_1, \ldots, t_n)$, as the right
  premise is valid. So $\val \M \Psi s t \in \val \M \Psi s
  {A(x)}$. Now let $s(z) = \val M \Psi s {\meps x A'(x, t_1, \ldots,
    t_n)}$. Then $\M, \Psi, s \models A(z)$ and so $\M, \Psi, s
  \not\models \lnot A(z)$. Since the left premise is valid, $\M, \Psi,
  s \models \Delta(z)$. But also $\M, \Psi, s \not\models \Delta(z)$
  since $\M, \Psi, s \not\models \Delta(\meps x A(x))$.}

Wessels offered a cut-elimination proof for her system. However, the
proof relied on a false lemma to which Maehara gave a counterexample.
\begin{center}
\textbf{Wessels' Lemma.} If $\vdash \Gamma, \Delta(\meps x A(x))$ then
$\vdash \Gamma, \Delta(z), \lnot A(z)$.
\end{center}

Let $A(x) = P(x, \meps y Q(\meps u P(u, y)))$, $\Delta(z) = Q(z)$, and
$\Gamma = \lnot Q(\meps x B(x, w))$.  Then
\[
\underbrace{\lnot Q(\meps x P(x, w))}_\Gamma, 
\underbrace{Q(\meps x P(x, \meps y Q(\meps u P(u, y))))}_{\Delta(\meps x A(x))}
\]
is derivable, since it is of the form $\lnot B(w), B(\meps y B(y))$.
However, the corresponding sequent in the consequent of the lemma,
\[
\underbrace{\lnot Q(\meps x P(x, w))}_\Gamma, 
\underbrace{Q(z)}_{\Delta(z)},
\underbrace{\lnot P(z, \meps y Q(\meps u P(u, y)))}_{\lnot A(z)}
\]
is not derivable, because not valid.\footnote{Let $\card{\M} = \{1,
  2\}$, $Q^\M = \{1\}, P^\M = \{\langle 1,2\rangle, \langle
  2,2\rangle\}$, $s(z) = s(w) = 2$.  Since $\langle 1, 2\rangle \in P^\M$, we
  can choose $\Psi$ so that $\val \M \Psi s {\meps x P(x, 2)} = 1$. So
  $\M, \Psi, s \not\models \lnot Q(\meps x P(x, w))$.  Also,
  $\M,\Psi,s \not\models Q(z)$.  As $\val \M \Psi s {\meps u P(u, 2)}
  = 1$ and $1 \in Q^\M$, we can also fix $\Psi$ so that $\val \M \Psi
  s {\meps y Q(\meps u P(u, y))} = 2$. But then $\M, \Psi, s
  \not\models \lnot P(z, \meps y Q(\meps u P(u, y)))$.)}

Mints (in a review of Wessels' paper) proposed the following rule instead:
\[
\infer[\eps1]{\Gamma, \Delta(\meps x A(x))}{
\Gamma, \Delta(\meps x A(x)), \lnot A(\meps x A(x)) & \Gamma, A(t) 
}
\]
It, too, derives all critical formulas:
\[
\infer[\eps1]{\lnot A(t), A(\meps x A(x))}{
  \infer[w]{\underbrace{\lnot A(t)}_\Gamma, \underbrace{A(\meps x A(x))}_\Delta, \lnot A(\meps x A(x))}{A(\meps x A(x)), \lnot A(\meps x A(x))} 
  &
  \underbrace{\lnot A(t)}_\Gamma, A(t)
}
\]
The system was developed in detail by Yasuhara \cite{Yasuhara:82}.
The Mints-Yasuhara system is cut-free complete.  However, it is not
known if the sequent has a cut-elimination theorem that transforms a
proof with cuts successively into one without cuts. Both Gentzen's and
Tait's approach to cut-elimination do not seem to work.  In a
Gentzen-style proof, the main induction is on on cut length, i.e., the
height of the proof tree above an uppermost cut.  In the induction
step, a cut is permuted upward to reduce the cut length. For instance,
we replace the subproof proof ending in a cut
\[
\infer[\textit{cut}]{\Pi, \Lambda, \exists x\, B(x)}{
  \infer*[\pi]{\Pi, \dem{A}}{} &
  \infer[\exists R]{\dem{\lnot A}, \Lambda, \exists x\, B(x)}{
    \infer*[\pi']{\lnot A, \Lambda, B(t)}{}
  }
}
\qquad\text{by}\qquad
\infer[\exists R]{\Pi, \Lambda, \exists x\, B(x)}{
\infer[\textit{cut}]{\Pi, \Lambda, B(t)}{
  \infer*[\pi]{\Pi, \dem{A}}{} &
  \infer*[\pi']{\dem{\lnot A}, \Lambda, B(t)}{}
  }
}
\]
To permute a cut across the $\eps 1$ rule:
\[
\infer[\textit{cut}]{\Pi, \Gamma, \Delta(\meps x B(x))}{
  \infer*[\pi]{\Pi, \dem{A}}{} &
  \infer[\eps 1]{\dem{\lnot A}, \Gamma, \Delta(\meps x B(x))}{
    \infer*[\pi']{\lnot A, \Gamma, \Delta(\meps x B(x)), \lnot B(\meps x B(x))}{}
    &
    \infer*[\pi'']{\Gamma, B(t)}{}
  }
}
\]
one might try to replace the proof tree with
\[
\infer[\eps 1]{\Gamma, \Delta(\meps x B(x))}{
\infer[\textit{cut}]{\Pi, \Gamma, \Delta(\meps x B(x)), \lnot B(\meps x B(x))}{
  \infer*[\pi]{\Pi, \dem{A}}{} &
    \infer*[\pi']{\dem{\lnot A}, \Gamma, \Delta(\meps x B(x)), \lnot B(\meps x B(x))}{}
} &
    \infer*[\pi'']{\Gamma, B(t)}{}
}
\]
However, here the condition on $\eps1$ is violated if $\lnot A$ is in
$\Delta$.

In a Tait-style cut elimination proof, the main induction is on cut
rank, i.e., complexity of the cut formula.  In the induction step, the
complexity of the cut formula is reduced.  For instance, if a subproof
ends in a cut
\[
\infer[\mathit{cut}]{\Pi, \Lambda}{
  \infer*[\pi]{\Pi, \dem{\lnot(A \land B)}}{}
  &
  \infer*[\pi']{\Lambda, \dem{A \land B\strut}}{}
}
\]
we replace it with
\[
\infer[\mathit{cut}]{\Pi, \Lambda}{
  \infer[\mathit{cut}]{\Pi, \Lambda, \dem{\lnot B}}{
    \infer*[\pi_1]{\Pi, \dem{\lnot A}, \lnot B}{}
    &
    \infer*[\pi'_1]{\Lambda, \dem{A}}{}
  }
  &
  \infer*[\pi'_2]{\Lambda, \dem{B}}{}
}
\] 
This approach requires \emph{inversion lemmas}.  A typical case is:
If $\pi' \vdash \Pi, A \land B$ then there is a $\pi_1' \vdash \Pi, A$
of cut rank and length $\le$ that of $\pi'$.  In the proof of the
inversion lemma, one replaces all ancestors of $A \land B$ in $\pi'$
by $A$ and ``fixes'' those rules that are no longer valid.  For
instance, replace
\[
\infer[\land R]{\Gamma, A}{
  \infer*{\Gamma, A}{} & \infer*{\Gamma, B}{}
}
\quad\text{by}
\quad
\infer*{\Gamma, A}{}
\]
But now consider a derivation $\pi'$ which contains the $\eps 1$
rule:\footnote{($A \land B(\meps x C(x))$ is $\Delta(\meps x C(x))$ in
  this case).}
\[
\infer[\eps 1]{\Pi, A \land B(\meps x C(x)}{
  \infer*{\Pi, A \land B(\meps x C(x)), \lnot C(\meps x C(x))}{}
  &
  \infer*{\Pi, C(t)}{}
}
\]
The inversion lemma produces
\[
\infer[\eps 1]{\Pi, A}{
  \infer*{\Pi, A, \lnot C(\meps x C(x))}{}
  &
  \infer*{\Pi, C(t)}{}
}
\]
This, again, no longer satisfies the condition of $\meps 1$.

\begin{prob}
Prove cut-elimination for the Mints-Yasuhara system, or give a
similarly simple sequent calculus for which it can be proved.
\end{prob}

\subsection{Natural deduction}

In Gentzen's classical natural deduction system NK, the quantifier
rules are given by
\begin{center}
  \begin{tabular}{c@{\qquad}c}
$\infer[\forall I]{\forall x\, A(x)}{A(z)}$ &
$\infer[\forall E]{A(t)}{\forall x\, A(x)}$ \\
$\infer[\exists I]{\exists x\, A(x)}{A(t)}$ &
$\infer[\exists E]{B}{\exists x\, A(x) & \infer*{C}{[A(z)]}}$
  \end{tabular}
  \end{center}
where $z$ must not appear in any undischarged assumptions (nor in
$A(x)$ or $B$). Meyer Viol \cite{MeyerViol:95} has proposed a system
in which the $\exists E$ rule is replaced by
\[
\infer[\exists E_\eps]{A(\meps x A(x))}{\exists x\, A(x)}
\]
and the following term rule is added
\[
\infer[I\eps]{A(\meps x A(x))}{A(t)}
\]

\begin{prob}
Does Meyer Viol's system have a normal form theorem?
\end{prob}

Adding $\exists E_\eps$ and $I\eps$ to the intuitionistic system NJ
results in a system that is not conservative over intuitionistic
logic. For instance, \emph{Plato's principle}, the formula
\[
\exists x(\exists y\, A(y) \to A(x))
\]
becomes derivable:
\[
\infer[\exists I]{\exists x(\exists y\, A(y) \to A(x))}{
  \infer[\to I]{\exists y\, A(y) \to A(\meps x A(x))}{
    \infer[\exists E_\eps]{A(\meps x A(x))}{[\exists y\, A(\meps x A(x))]}
    }}
\]
However, the system also does not collapse to classical logic: it is
conservative for propositional formulas.

Intuitionistic natural deduction systems are especially intriguing, as
Abadi, Gonthier and Werner \cite{AbadiGonthierWerner:04} have shown
that a system of quantified propositional intuitionistic logic with a
choice operator $\meps X$ can be given a Curry-Howard correspondence
via a type system which $\meps X A(X)$ is a type such that the type
$A(X)$ is inhabited.  System $\cal E$ is paired with a simply typed
$\lambda$-calculus that, in addition to $\lambda$-abstraction and
application, features \emph{implementation}: \( \langle t\colon A
\text{ with } X = T\rangle \) of type $A(\epsilon_X A/X)$. If $A(X)$
is a type specification of an interface with variable type $X$, then
$A(T)$ for some type $T$ is an implementation of that interface.

\nocite{Mancosu:98}
\bibliographystyle{splncs03}
\bibliography{epsilon}
\end{document}